\documentclass[a4paper]{amsart}
\usepackage[latin1]{inputenc}
\usepackage{amssymb}
\usepackage{amsmath}
\usepackage{mathrsfs}
\usepackage{eufrak}
\usepackage{amsthm}
\usepackage{amsfonts}
\usepackage{textcomp}
\usepackage{graphicx}
\usepackage[pdftex]{color}
\usepackage{paralist}
\usepackage[shortlabels]{enumitem}
\usepackage{hyperref}
\usepackage{comment}
\usepackage[arrow, matrix, curve]{xy}

\newtheorem*{corollary*}{Corollary}
\newtheorem*{example*}{Example}
\newtheorem*{theorem*}{Theorem}
\newtheorem*{proposition*}{Proposition}

\newtheorem{theorem}{Theorem}[section]

\newtheorem{lemma}[theorem]{Lemma}
\newtheorem{proposition}[theorem]{Proposition}

\newtheorem*{claim*}{Claim}

\theoremstyle{definition}

\theoremstyle{remark}

\numberwithin{equation}{theorem}

\makeatletter
\renewcommand*\env@matrix[1][\
arraystretch]{%
  \edef\arraystretch{#1}%
  \hskip -\arraycolsep
  \let\@ifnextchar\new@ifnextchar
  \array{*\c@MaxMatrixCols c}}
\makeatother

\begin{document}

\title{On a new formula for the Gorenstein dimension}
\date{\today}



\author{Ren\'{e} Marczinzik}
\address{Institute of algebra and number theory, University of Stuttgart, Pfaffenwaldring 57, 70569 Stuttgart, Germany}
\email{marczire@mathematik.uni-stuttgart.de}

\begin{abstract}
Let $A$ be a finite dimensional algebra over a field $K$ with enveloping algebra $A^e=A^{op} \otimes_K A$.
We call algebras $A$ that have the property that the subcategory of Gorenstein projective modules in $mod-A$ coincide with the subcategory $\{ X \in mod-A | Ext_A^i(X,A)=0 $ for all $i \geq 1 \}$ left nearly Gorenstein. The class of left nearly Gorenstein algebras is a large class that includes for example all Gorenstein algebras and all representation-finite algebras. We prove that the Gorenstein dimension of $A$ coincides with the Gorenstein projective dimension of the regular module as $A^e$-module for left nearly Gorenstein algebras $A$.
We give three application of this result. The first generalises a formula by Happel for the global dimension of algebras. The second applications generalises a criterion of Shen for an algebra to be selfinjective. As a final application we prove a stronger version of the first Tachikawa conjecture for left nearly Gorenstein algebras.

\end{abstract}

\maketitle
\section*{Introduction}
Let $A$ always be a finite dimensional connected algebra over a field $K$ and modules are finite dimensional right modules if nothing is stated otherwise. $A^e:=A^{op} \otimes_K A$ denotes the enveloping algebra of $A$. It is well known that the $A^e$-modules correspond to the $A$-bimodules.
In \cite{Ha}, Happel obtained the formula $pd_{A^e}(A)=gldim(A)$ for finite dimensional algebras over an algebraically closed field as a corollary of a result about the minimal projective resolution of the regular module as a bimodule.
This can be used for example to show that the Hochschild cohomology of algebras of finite global dimension $n$ vanishes for degrees greater than $n$ over an algebraically closed field.
However, Happel's formula is not correct over general fields as the next example shows, where $F_p$ denotes the finite field with $p$ elements:
\begin{example*} \label{fexample}
Let $k:=F_p(X)$ and $K:=k[T]/(T^p-X)$. Then $K$ is a finite field extension of $k$ and one has $K \otimes_k K \cong K[T]/(T^p)$. Thus the projective dimension of $K$ over $K \otimes_k K$ is infinite, while the global dimension of $K$ is zero, since it is a field. 
\end{example*}

Thus it is desirable to obtain a formula valid for algebras over any fields.
We achieve this goal using Gorenstein homological algebra.
One goal of Gorenstein homological algebra is to extent classical results from homological algebra by a Gorenstein homological version, see for example the meta-question in the introduction of \cite{CFH}.
The Gorenstein homological version of finite global dimension $gldim(A)$ is finite Gorenstein dimension $Gdim(A)$ and the finite projective dimension of a module $M$  $pd_A(M)$ corresponds to finite Gorenstein projective dimension $Gpd_A(M)$.
Following \cite{Mar}, we call an algebra left nearly Gorenstein in case the full subcategory of Gorenstein projective modules coincides with the full subcategory of modules $M$ with $Ext^i(M,A)=0$ for all $i \geq 1$.
The class of left nearly Gorenstein algebras is very large. It contains for example the class of Gorenstein algebras and all representation-finite algebras.
In fact, only around 40 years after the definition of Gorenstein projective modules by Auslander and Bridger in \cite{AB} an algebra that is not left nearly Gorenstein was found in \cite{JS}. For a systematic construction of algebras that are not left nearly Gorenstein we refer to \cite{Mar2}.
Our main result is as follows:

\begin{theorem*}
Let $A$ be a left nearly Gorenstein algebra.
Then $Gdim(A)=Gpd_{A^e}(A)$.
\end{theorem*}

We obtain three applications of the the theorem.
The first generalises the result of Happel to arbitrary algebras of finite global dimension, without restrictions on the field.

\begin{proposition*}
Let $A$ be an algebra of finite global dimension.
Then $gldim(A)=Gpd_{A^e}(A)$.
\end{proposition*}

The next application generalises a characterisation of selfinjective algebras from \cite{S} proposition 5.3., where the same result was proven with the stronger assumption that the algebra is Gorenstein:

\begin{proposition*}
Let $A$ be a left nearly Gorenstein algebra. Then $A$ is selfinjective iff $A$ is Gorenstein projective as an $A^e$-module.
\end{proposition*}

Recall that the first Tachikawa conjecture states that $Ext^i(D(A),A) \neq 0$ for some $i \geq 1$ for any non-selfinjective algebra $A$, see \cite{Ta}. Note that the first Tachikawa conjecture is trivially true for Gorenstein algebras with positive Gorenstein dimension $g$, since $Ext^g(D(A),A) \neq 0$ because $D(A)$ has projective dimension $g$.
We prove a much stronger version of the first Tachikawa conjecture for left nearly Gorenstein algebras:
\begin{proposition*}
Let $A$ be a left nearly Gorenstein algebra that is not Gorenstein.
Then $Ext^i(D(A),A) \neq 0$ for infinitely many $i \geq 1$.
\end{proposition*}
We assume the reader is familiar with the basics of representation theory and homological algebra of finite dimensional algebras as explained for example in the book \cite{SkoYam}. For the basics on Gorenstein homological algebra for finite dimensional algebras we refer to \cite{Che}.
The author is thankful to Dawei Shen for useful comments and spotting an error in an earlier version. Proposition 2.2. of this article is due to Dawei Shen. The author is thankful to Matthew Pressland for useful discussions.
\section{Preliminaries}

Recall that an algebra $A$ is called \emph{Gorenstein} in case the left and right regular modules have finite injective dimensions and those dimensions coincide.
A complex $\cdots P_1 \xrightarrow{d_1} P_0 \xrightarrow{d_0} P_{-1} \rightarrow \cdots$ of (possibly infinite dimensional) projective modules $P_i$ is called \emph{totally acyclic} in case the complex is acyclic and the complex remains acyclic when the functor $Hom_A(-,Q)$ is applied for any (possibly infinite dimensional) projective module $Q$. A (possibly infinite dimensional) module $M$ is called \emph{Gorenstein projective} in case it is the kernel of a map $d_0$ in a totally acyclic complex as above. The \emph{Gorenstein projective dimension} $Gpd_A(M)$ of an $A$-module $M$ is defined as follows: $Gpd_A(M) \leq n$ for $n \geq 0$ iff there is an exact sequence of the form $0 \rightarrow G^{-n} \rightarrow \cdots \rightarrow G^{-1} \rightarrow G^0 \rightarrow M \rightarrow 0$ for (possibly infinite dimensional) Gorenstein projective modules $G^{-i}$.
One can show that an algebra is Gorenstein iff every module has finite Gorenstein projective dimension, see corollary 3.2.6. of \cite{Che}.
We refer to \cite{Che} for more characterisations of Gorenstein modules and the Gorenstein projective dimension. Note that algebras of finite global dimension are Gorenstein algebras and the global dimension coincides with the Gorenstein dimension in this case. In the case of finite global dimension all Gorenstein projective modules are projective and the Gorenstein projective dimension coincides with the usual projective dimension.
We note that the only occurence of infinite dimensional modules is in the calculation of Gorenstein projective dimensions in this article.
If not explicitly noted all modules are finite dimensional. 
An $A$-module $M$ is called a \emph{stable module} in case $Ext^i(M,A)=0$ for all $i>0$.
Following \cite{Mar}, an algebra $A$ is called \emph{left nearly Gorenstein} in case the full subcategory of Gorenstein projective modules coincides with the full subcategory of stable modules.
For example all Gorenstein algebras and all representation finite algebras are left nearly Gorenstein, see \cite{Mar} for more examples and information.

The famous Gorenstein symmetry conjecture, see for example \cite{ARS} in the section about conjectures, states that the left and the right injective dimension of the regular module always coincide. The Gorenstein symmetry conjecture is open in general but true for left nearly Gorenstein algebras, see for example \cite{Mar} for a short proof.

\begin{proposition} \label{chprop}
Let $M$ be an $A$-module, then $Gpd_A(M) \geq \sup \{ i \geq 0 | Ext_A^i(M,A) \neq 0 \}$ with equality in case $Gpd_A(M)$ is finite.
\end{proposition}
\begin{proof}
In case $Gpd_A(M)$ is infinite there is nothing to show. In case $Gpd_A(M)$ is finite one has $Gpd_A(M) = \sup \{ i \geq 0 | Ext_A^i(M,A) \neq 0 \}$, this is contained in proposition 3.2.2. of \cite{Che}.
\end{proof}

\section{Main results}

\begin{lemma} \label{lemmata}
We have $Ext_A^i(D(A),A) \cong Ext_{A^e}^i(A, A^{e})$ for all $i \geq 1$.
\end{lemma}
\begin{proof}
On page 114 of \cite{Ta}, Tachikawa proves that $Ext_A^i(D(A),A) \cong Ext_{A^e}^i(A, A \otimes_K A)$ for all $i \geq 1$. Now note that as $A$-bimodules we have $A \otimes_K A \cong A^e$.
Thus $Ext_A^i(D(A),A) \cong Ext_{A^e}^i(A, A \otimes_K A) \cong Ext_{A^e}^i(A, A^e)$ for all $i \geq 1$.
\end{proof}

The following proposition and its proof are due to Dawei Shen:
\begin{proposition} \label{injgorproprop}
Let $I$ be an injective module of infinite projective dimension. Then there is no $m \geq 1$ such that $\Omega^m(I)$ is Gorenstein projective.
\end{proposition}
\begin{proof}
We denote by $Gp^{\perp}$ the full subcategory $Gp^{\perp}:=\{ M | Ext^i(G,M)=0$ for all $i >0$ and any Gorenstein projective module $G \}$. Note that this coincides with $\{ M | Ext^1(G,M)=0$ for any Gorenstein projective module $G \}$ since with $G$ also $\Omega^r(G)$ is Gorenstein projective for any $r \geq 1$. Note that $Gp^{\perp}$ contains all projective and all injective modules.
We show that $Gp^{\perp}$ is also closed under kernels of surjections:
Assume there is a short exact sequence of modules
$$0 \rightarrow X \rightarrow X_1 \rightarrow X_2 \rightarrow 0 \ (*)$$
 with $X_1, X_2 \in Gp^{\perp}$.
Let $G$ be a Gorenstein projective module. By definition of Gorenstein projective modules there is a short exact sequence $0 \rightarrow G \rightarrow P \rightarrow G_1 \rightarrow 0$ with $G_1$ also Gorenstein projective and $P$ being projective.
Applying the functor $Hom_A(G_1,-)$ to the short exact sequence $(*)$ and noting that $X_1, X_2 \in Gp^{\perp}$, we obtain the exact sequence:
$$0 \rightarrow Hom_A(G_1,X) \rightarrow Hom_A(G_1,X_1) \rightarrow Hom_A(G_1,X_2) \rightarrow Ext^1(G_1,X) \rightarrow 0 \rightarrow 0 \rightarrow Ext^2(G_1,X) \rightarrow 0 \cdots$$
This shows that $Ext^2(G_1,X)=0$ and thus also $Ext^1(G,X)=Ext^1(\Omega^1(G_1),X)=Ext^2(G_1,X)=0$.
Thus $X \in Gp^{\perp}$.

Now if $I$ is as in the proposition, we have $\Omega^r(I) \in Gp^{\perp}$ for all $ r \geq 0$ by induction, using that there are short exact sequences 
$$0 \rightarrow \Omega^r(D(A)) \rightarrow P_r \rightarrow \Omega^{r-1}(D(A)) \rightarrow 0,$$
with $P_r$ projective. In case $\Omega^r(I)$ is Gorenstein projective, we have that $\Omega^r(I)$ is in fact projective since all Gorenstein projective modules in $Gp^{\perp}$ are projective, see for example appendix A of \cite{Che}.
But $\Omega^r(I)$ being projective contradicts our assumption that $I$ has infinite projective dimension.

\end{proof}

\begin{theorem} \label{mainresult}
Let $A$ be a left nearly Gorenstein algebra, then $Gdim(A)=Gpd_{A^e}(A)= \sup \{ i \geq 0 | Ext_A^i(D(A),A) \neq 0 \}=Gpd_A(D(A))$.
\end{theorem}
\begin{proof}
First note that $A$ satisfies the Gorenstein symmetry conjecture since $A$ is assumed to be left nearly Gorenstein as was mentioned earlier. \newline
\underline{Case 1:} First assume that $A$ is Gorenstein. Then also $A^e$ is Gorenstein and the Gorenstein dimension of $A$ is equal to the projective dimension of $D(A)$ which is equal to $\sup \{ i \geq 0 | Ext_A^i(D(A),A) \neq 0 \}$, since for any module $M$ of finite projective dimension the projective dimension is equal to $\sup \{ i \geq 0 | Ext_A^i(M,A) \neq 0 \}$, see for example \cite{ARS} VI. lemma 5.5. Now since $A^e$ is Gorenstein, the Gorenstein projective dimension of every $A^e$-module is finite and by \ref{chprop}, $Gpd_{A^e}(A)= \sup \{ i \geq 0 | Ext_{A^e}^i(A, A^{e})\}$. Using \ref{lemmata}, this gives the theorem for $A$ being Gorenstein. \newline
\underline{Case 2:} Now assume that $A$ has infinite Gorenstein dimension. 
We proof that $Ext^i(D(A),A) \neq 0$ for infinitely many $i$ which gives the result, because $Gpd_{A^e}(A) \geq \sup \{ i \geq 0 | Ext_{A^e}^i(A, A^{e}) \neq 0 \} = \sup \{ i \geq 0 | Ext_A^i(D(A),A) \neq 0 \}$ and $Gpd_A(D(A)) \geq \sup \{ i \geq 0 | Ext_A^i(D(A),A) \neq 0 \}$ by \ref{lemmata} and \ref{chprop}.
Now assume that $Ext^i(D(A),A) \neq 0$ for only finitely many $i$.
Then there is a natural number $r$ such that $Ext^{r+i}(D(A),A) =0$ for all $i \geq 1$. This shows that $Ext^i(\Omega^r(D(A)),A)=Ext^{r+i}(D(A),A) =0$ for all $i \geq 1$. Thus $\Omega^r(D(A))$ is a stable module and thus Gorenstein projective since we assume that $A$ is left nearly Gorenstein. But by \ref{injgorproprop} this is a contradiction since $D(A)$ and thus $\Omega^r(D(A))$ have infinite projective dimension since we assume that $A$ is not Gorenstein. Thus we have $Ext^i(D(A),A) \neq 0$ for infinitely many $i$ and this shows the result.
\end{proof}
We remark that the proof of the theorem shows that the result holds for any algebras with $Ext^i(D(A),A) \neq 0$ for infinitely many $i \geq 1$.
We note this as a proposition:
\begin{proposition}
Let $A$ be an algebra with $Ext^i(D(A),A) \neq 0$ for infinitely many $i \geq 1$ then $Gdim(A)=Gpd_{A^e}(A)= \sup \{ i \geq 0 | Ext_A^i(D(A),A) \neq 0 \}=Gpd_A(D(A))$.
\end{proposition}
This can be used to prove $Gdim(A)=Gpd_{A^e}(A)$ for algebras that are not left nearly Gorenstein. For example the commutative algebras in \cite{JS} are not left nearly Gorenstein, but computer experiments suggest that $Ext^i(D(A),A) \neq 0$ for infinitely many $i \geq 1$ for those algebras. We note that we are not aware of a commutative algebra that is not Gorenstein with $Ext^i(D(A),A) \neq 0$ for only finitely many $i \geq 1$.

We give 3 applications.
The first application generalises Happel's formula to algebras over fields that are not necessarily algebraically closed:
\begin{proposition}
Let $A$ be an algebra of finite global dimension, then $Gpd_{A^e}(A)=gldim(A)$.
\end{proposition}
\begin{proof}
This is a special case of \ref{mainresult} by noting that the global dimension coincides with the Gorenstein dimension in case the global dimension is finite.
\end{proof}
We remark that the previous proposition gives the correct result for the example in \ref{fexample}, since there $K \otimes_k K$ is a selfinjective algebra and thus the Gorenstein projective dimension of any module is zero because every module in a selfinjective algebra is Gorenstein projective.
The next application gives a characterisation when the regular module is Gorenstein-projective as a bimodule for left nearly Gorenstein algebras. This generalises proposition 5.3. of \cite{S}.
\begin{proposition}
Let $A$ be left nearly Gorenstein.
$A$ is a Gorenstein projective $A^e$-module iff $A$ is selfinjective.
\end{proposition}
\begin{proof}
$A$ being Gorenstein projective as an $A^e$-module is equivalent that $A$ as an $A^e$-module has Gorenstein projective dimension zero. Thus by \ref{mainresult} $A$ is Gorenstein projective as an $A^e$-module iff $A$ has Gorenstein dimension zero, which is equivalent to $A$ being selfinjective.
\end{proof}

In \cite{Ta}, Tachikawa conjectured that for an algebra $A$ one has $Ext_A^i(D(A),A) =0$ for all $i \geq 1$ iff $A$ is selfinjective. This is now known as the first Tachikawa conjecture. The second Tachikawa conjectures states that $Ext_B^i(M,M) \neq 0$ for some $i \geq 1$ for any non-projective module $M$ over a selfinjective algebra $B$. The first and second Tachikawa conjecture are together equivalent to the famous Nakayama conjecture.
We refer to \cite{Yam} for more on those homological conjectures.
The first Tachikawa conjecture is trivial for Gorenstein algebras of Gorenstein dimension $g \geq 1$, since in this case $Ext^g(D(A),A) \neq 0$.
The following shows that a much stronger version of the first Tachikawa conjecture is true for left nearly Gorenstein algebras:
\begin{proposition}
Let $A$ be a left nearly Gorenstein algebra.
Then $Ext^i(D(A),A) \neq 0$ for infinitely many $i \geq 1$ in case $A$ is not Gorenstein.
\end{proposition}
\begin{proof}
This is a direct consequence of \ref{mainresult} for non Gorenstein algebras.
\end{proof}

We end the article with some questions:
\begin{enumerate}
\item Does $Gdim(A)=Gpd_{A^e}(A)$ hold for any finite dimensional algebra $A$?
\item Can one give an explicit resolution by Gorenstein projective modules of $A$ as an $A^e$-module in special cases, such as for Gorenstein algebras?
\item If two algebras $A$ and $B$ are left nearly Gorenstein, is also their tensor product left nearly Gorenstein? What about the other direction: If $A \otimes_K B$ is left nearly Gorenstein, are $A$ and $B$ left nearly Gorenstein?
\item Given a commutative algebra that is not Gorenstein, do we have $Ext^i(D(A),A) \neq 0$ for infinitely many $i \geq 1$?
\end{enumerate}

\end{document}